\documentclass{amsart}

\usepackage{amsmath,amssymb}

\newtheorem{thm}{Theorem}[section]
\newtheorem{cor}[thm]{Corollary}
\newtheorem{lem}[thm]{Lemma}
\newtheorem{prop}[thm]{Proposition}

\newtheorem*{thm*}{Theorem}
\newtheorem*{thma}{Theorem A}
\newtheorem*{thmb}{Theorem B}

\theoremstyle{definition}
\newtheorem{defn}[thm]{Definition}
\newtheorem{rem}[thm]{Remark}
\newtheorem{exa}[thm]{Example}

\usepackage[utf8]{inputenc}
\usepackage[T1]{fontenc}

\begin{document}

\title{Log canonical thresholds on group compactifications}

\author[T. Delcroix]{Thibaut Delcroix}
\email{thibaut.delcroix@ujf-grenoble.fr}
\address{Univ. Grenoble Alpes, IF, F-38000 Grenoble, France \\ 
CNRS, IF, F-38000 Grenoble, France}


\begin{abstract}
We compute the log canonical thresholds of non-negatively curved singular hermitian metrics on ample linearized line bundles 
on bi-equivariant group compactifications of complex reductive groups. To this end, we associate to any such metric a convex 
function whose asymptotic behavior determines the log canonical threshold. As a consequence we obtain a formula for the 
alpha invariant of these line bundles, in terms of the polytope associated to the group compactification.
\end{abstract}

\maketitle

\section*{Introduction}

The aim of this article is to begin the study of Kähler metrics on polarized 
$G\times G$-equivariant compactifications of a connected complex reductive group $G$. 
This class of manifolds generalizes the well known class of polarized toric manifolds, 
and we extend some techniques of toric geometry to this setting.
In this article, we study singular hermitian metrics on linearized ample line bundles,
which are non-negatively curved and invariant under the action of $K\times K$, where 
$K$ is a maximal compact subgroup of $G$. We associate to such a metric a convex 
function on some real vector space, that we call the convex potential of the metric, 
and show how the asymptotic behavior of this function is controlled by a polytope 
associated to the line bundle. This generalizes the case of toric manifolds and relies 
on the $KAK$ decomposition of a reductive group.

The correspondence between metrics and their convex potentials is a bijection and 
provides a description of the set of non-negatively curved, $K\times K$-invariant, 
singular hermitian metrics. Furthermore, to obtain this description we use a  
special continuous reference metric that generalizes the Batyrev-Tschinkel metric 
on toric line bundles, which had already been used as a model for the behavior of 
continuous metrics in \cite{CLT10}.

We then proceed to compute the log canonical threshold of such metrics, in terms of 
the asymptotic behavior of their convex potential. To achieve that goal, we associate 
a convex body to the metric, that we call the Newton body of the metric, 
which gives another way to encode the asymptotic behavior 
of the convex potential, and is well suited to fan decompositions. We should stress at this 
point that another important ingredient is the existence of a toric subvariety (and 
corresponding fan) in any group compactification, that contains the information about 
the compactification. We obtain the following theorem.

\begin{thma}
Let $(X,L)$ be a polarized $G\times G$-equivariant compactification of $G$.
Assume furthermore that $X$ is Fano.
Denote by $P$ the polytope associated to $L$ and by $Q$ the polytope associated 
to the anticanonical line bundle $-K_X$. Let also $H$ denote the convex hull of the 
images, by the Weyl group $W$ of $G$, of the sum of the positive roots of $G$.
Let $h$ be a $K\times K$-invariant hermitian metric with non negative curvature on $L$, 
then the log canonical threshold of $h$ is given by:
\[
\mathrm{lct}(h)=\mathrm{sup}\{c>0; 2H+2cP \subset cN(h)+2Q\},
\]
where $N(h)$ is the Newton body of $h$.
\end{thma}
 
Using this expression of the log canonical threshold, we are able to compute Tian's $\alpha$
invariant for any ample linearized line bundle on a Fano group compactification, with 
respect to the $K\times K$ action.

\begin{thmb}
Let $(X,L)$ be a polarized compactification of $G$, and $P:=P(X,L)$.
Assume furthermore that $X$ is Fano and let $Q:=P(X,-K_X)$.
Then 
\[
\alpha_{K\times K}(L)=\mathrm{sup} \{ c>0 ; c(P + (-P^W)) \subset Q \ominus H \},
\]
where $P^W$ denotes the subset of $W$-invariant points of $P$ and $W$ is the 
Weyl group of $G$.
\end{thmb}

This formula generalizes the formula for the $\alpha$-invariant of polarized Fano toric 
manifolds previously obtained in \cite{Del14}, 
and independently by other authors \cite{Amb14,LSY15}. In the case of the anticanonical 
line bundle on a Fano toric manifold, the formula was initially obtained by Song \cite{Son05}.

The original motivation of this work was to obtain such an expression, hoping that Tian's 
criterion for the existence of Kähler-Einstein metrics in terms of this invariant would 
be satisfied by some group compactifications. Recall that Tian's criterion \cite{Tia87} 
is that if the 
$\alpha$ invariant is strictly greater than $\frac{n}{n+1}$ where $n$ is the dimension 
of the manifold, then there exists a Kähler-Einstein metric. 
 
For toric manifolds, additional symmetries had to be taken into account for the criterion
to be satisfied. 
For the important examples of wonderful compactifications of semisimple adjoint groups
with no rank one factor, Brion computed their automorphism group in \cite{Bri07}, so that
our result allows to compute the alpha invariant with respect to a maximal compact 
subgroup of the automorphism group, but Tian's criterion is not satisfied. 
Even though we can in some cases consider additional symmetries, we do not obtain new 
examples of Kähler-Einstein metrics by this method. 
Our study of hermitian metrics will be used in \cite{DelKE} where we obtain 
a necessary and sufficient condition for the existence of Kähler-Einstein metrics 
on a group compactification in terms 
of the polytope. The present article and \cite{DelKE} contain the main results of the 
author's PhD thesis \cite{DelTh}.

\section{Group compactifications}

\subsection{Definition and examples}

Let $G$ be a connected complex reductive group.

\begin{defn}
Let $X$ be a projective manifold. We say that $X$ is a 
\emph{smooth $G\times G$-equivariant compactification of $G$}
(or in short a \emph{compactification of $G$}) if $X$ admits a holomorphic 
$G\times G$-action with an open and dense orbit equivariantly isomorphic to $G$ as 
a $G\times G$-homogeneous space under the action defined by 
$(g_1,g_2)\cdot g= g_1gg_2^{-1}$.
\end{defn}

Let $X$ be a compactification of $G$.
We will always identify $G$ with the open and dense orbit in $X$. These manifolds 
belong to the class of \emph{spherical manifolds} \cite{Per14,Tim11}. 
There is a finite number of 
$G\times G$-orbits in $X$ and the boundary $X\setminus G$ is of codimension one.

\begin{exa}
If $G=T \simeq (\mathbb{C}^*)^n$ is a torus, then the compactifications of $T$ are the  
projective toric manifolds. One goes from the $T$-action to the $T\times T$ action 
through the morphism $T\times T \rightarrow T, (t_1,t_2)\mapsto t_1t_2^{-1}$.
\end{exa}

\begin{exa}
Assume that $G$ is an adjoint semisimple group. 
Then De Concini and Procesi \cite{DCP83} showed the existence of a special compactification
of $G$, called the wonderful compactification of $G$.
It is the only compactification of $G$ satisfying the following property : its boundary 
$X\setminus G$ (where we identify 
the open and dense orbit in $X$ with $G$) is a union of simple normal crossing prime 
divisors $D_i$, $i\in \{1,\ldots, r\}$, such that for any subset 
$I\subset \{1,\ldots, r\}$,
the intersection $X\cap \bigcap_{i\in I} D_i$ is the closure of a unique $G\times G$-orbit, and 
all $G\times G$-orbits appear this way. 
The integer $r$ is equal to the rank of $G$, which is the dimension of a maximal 
torus in $G$.

The wonderful compactification of $\mathrm{PGL}_2(\mathbb{C})$ is especially 
simple : it is $\mathbb{P}^3$ considered as $\mathbb{P}(\mathrm{Mat}_{2,2}(\mathbb{C}))$
equipped with the action of $\mathrm{PGL}_2(\mathbb{C})\times \mathrm{PGL}_2(\mathbb{C})$ 
induced by the multiplication of matrices on the left and on the right.
\end{exa}

\subsection{Polytopes associated to a polarized group compactification}

Recall that a \emph{$G$-linearized line bundle} over a $G$-manifold $X$ is a 
line bundle $L$ on $X$ equipped with an action of $G$ lifting the action on $X$, 
and such that the morphisms between the fibers induced by this action 
are linear. 

Let $G$ be a connected complex reductive group.
We call a \emph{polarized group compactification} a pair $(X,L)$ where $X$ is a 
compactification of $G$ and $L$ is a $G\times G$-linearized line bundle on $X$.

Choose $T$ a maximal torus in $G$, and let $S$ be a maximal compact torus in $T$. 
We denote, as usual, by $M$ the lattice of characters of $T$ and by $N$ the lattice 
of one parameter subgroups of $T$, naturally dual to each other. 
Denote by $\mathfrak{s}$ the Lie algebra of $S$, by $\mathfrak{t}$ the Lie algebra 
of $T$ and by $\mathfrak{a}$ the Lie algebra $i \mathfrak{s} \subset \mathfrak{t}$. 
We identify $\mathfrak{a}$ with $N\otimes \mathbb{R}$, and $\mathfrak{a}^*$ with 
$M\otimes \mathbb{R}$.

Let $\Phi\subset \mathfrak{a}^*$ denote the root system of $(G,T)$. Let $W$ be 
its Weyl group. Choose a system of positive roots $\Phi^+$. It defines a positive 
Weyl chamber $\mathfrak{a}^+$ in $\mathfrak{a}$, resp. $\mathfrak{a}^*_+$ in 
$\mathfrak{a}^*$.

\begin{thm}{\cite[Section 2]{AB04II}}  
Let $(X,L)$ be a  polarized group compactification of $G$. Denote by $Z$ the 
closure of $T$ in $X$. Then $Z$ is a toric manifold, equipped with a $W$-action, 
and $L|_Z$ is a $W$-linearized ample toric line bundle on $Z$.
\end{thm}

We denote by $P(X,L)$, or $P$ for simplicity, the polytope associated to the ample 
toric line bundle $L|_Z$ by the theory of toric varieties \cite{Ful93,Oda88}. 
The polytope $P$ is a lattice polytope in $M\otimes \mathbb{R}$, and it is 
$W$-invariant. Define $P^+(X,L)=P(X,L)\cap \mathfrak{a}^*_+$. It is a polytope in $\mathfrak{a}^*$, and $P(X,L)$ is the union of the images of $P^+(X,L)$ by $W$. 

The polytope $P^+(X,L)$ encodes the structure of $G\times G$-representation of the 
space of holomorphic sections of $L$, generalizing the same property for toric line 
bundles.
\begin{prop}{\cite[Section 2.2]{AB04II}}   
Let $(X,L)$  be a polarized group compactification, then 
\[
H^0(X,L)\simeq \bigoplus \{  \mathrm{End}(V_{\alpha}) ; \alpha \in M\cap P^+(X,L)\}
\]
where $V_{\alpha}$ is an irreducible representation of $G$ with highest weight $\alpha$.
\end{prop}

\begin{exa}{\cite[Proposition 6.1.11]{BK05}}
The wonderful compactification $X$ of an adjoint semisimple group is Fano.
The corresponding polytope $P(X,-K_X)$ is the convex hull of the images by 
the Weyl group $W$ of the weight $2\rho + \sum_{i=1}^r\alpha_i$, where 
the $\alpha_i$ are the simple roots of $\Phi^+$ and $2\rho$ is the sum of the 
positive roots.
\end{exa}

\section{Convex potential}

In this section we introduce the convex potential of a $K\times K$-invariant, 
non-negatively curved singular hermitian metric on a polarized group 
compactification. This correspondence gives a bijection between the set of 
these metrics and the set of the $W$-invariant convex functions on $\mathfrak{a}$
which satisfy asymptotic behavior conditions.

\subsection{Singular hermitian metrics and potentials}

Let $X$ be a compactification of $G$, and $L$ a linearized ample line bundle on $X$.
Given a hermitian metric $h$ on $L$ and a local trivialization $s$ of $L$ on an open 
subset $U\subset X$, the \emph{local potential} of $h$ with respect to $s$ is the 
function $\phi$ defined on $U$ by 
\[
\phi(x)=-\ln(|s(x)|_h^2).
\]
We consider here singular hermitian metrics and only require that the potential with 
respect to any local trivialization is locally integrable. The value $+\infty$ for 
the potentials is allowed.

We say that a hermitian metric is \emph{locally bounded} if its potentials with 
respect to any trivialization on a sufficiently small open subset are bounded. 
A hermitian metric is smooth (resp. continuous) if and only if its potentials with 
respect to any local trivialization are. A continuous hermitian metric is locally 
bounded. 

Given a reference metric $h_0$, we define the \emph{global potential} $\psi$ of $h$ 
with respect to $h_0$ by 
\[
\psi(x)= -\ln \left( \frac{|y |_h^2}{|y |_{h_0}^2} \right),
\]
for  any element $y$ of the fiber $L_x$. This is a function on $X$ that can a priori 
take the values $\pm \infty$. If $s$ is a local trivialization on $U$, $\phi$ 
(resp. $\phi_0$) is the potential of $h$ (resp. $h_0$) with respect to $s$, then 
$\psi = \phi - \phi_0$ on $U$.

\subsection{The convex potential}

Let $(X,L)$ be a polarized compactification of $G$.

We identify $G$ with the open dense orbit in $X$, and first build a trivialization 
of $L$ on $G$. Choose $1_e$ a non-zero element of the fiber $L_e$ above the neutral 
element $e$ of $G$. Define the section $s$ of $L$ on $G$ by $s(g)=(g,e)\cdot 1_e$.
This section is a trivialization over $G$, equivariant under the action of 
$G\times \{e\}$, and any such trivialization is a scalar multiple of $s$. 

Denote by $\phi$ the local potential of $h$ on $G$ with respect to $s$. 

We are interested in hermitian metrics that are invariant under the action of 
$K\times K$ where $K$ is a maximal compact subgroup of $G$. We choose $K$ such that 
$K\cap T = S$. We will use the classical $KAK$ decomposition of a complex reductive 
group.

\begin{prop}{\cite[Theorem 7.39]{Kna02}}
Any element $g\in G$ can be written in the form $g=k_1 \exp(x) k_2$ where 
$k_1, k_2\in K$ and $x\in \mathfrak{a}^+$. Furthermore, $x$ is uniquely determined 
by $g$. In other words, the set $A^+:= \{\exp(x) ; x\in \mathfrak{a}^+\}$ is a 
fundamental domain for the $K\times K$-action on $G$.
\end{prop}

We first remark that if $h$ is $K\times K$-invariant, then its potential with respect 
to the section constructed above is still $K\times K$-invariant.

\begin{prop}
Assume that $h$ is $K\times K$-invariant, then $\phi$ is also $K\times K$-invariant. 
\end{prop}

\begin{proof}   
Let $k_1, k_2 \in K$ and $g\in G$.
We can first write 
\begin{align*}
s(k_1 g k_2) & = (k_1 g k_2,e)\cdot 1_e \\
	& = (k_1, k_2^{-1}) (g,e) (k_2,k_2) \cdot 1_e.
\end{align*}
The subgroup $\mathrm{diag}(G)=\{(g,g)|g\in G\}$ fixes the neutral element $e\in G$, 
and thus acts on the fiber $L_e$ through a character $\chi$ of $G$, so that 
$(g,g)\cdot 1_e= \chi(g) 1_e$.

We can thus compute 
\begin{align*}
\phi(k_1 g k_2) & = -\ln(|s(k_1 g k_2)|_h^2) \\
	& = -\ln(|(k_1, k_2^{-1}) (g,e) (k_2,k_2) \cdot 1_e|_h^2) \\
\intertext{by $K\times K$-invariance of $h$, this is}
	& =  -\ln(|(g,e) (k_2,k_2) \cdot 1_e|_h^2) \\
\intertext{and, by linearity,}
	& = -\ln(|\chi(k_2)| |(g,e) \cdot 1_e|_h^2) \\
\intertext{Since $K$ is compact, $|\chi(k_2)|=1$, so we obtain}
\phi(k_1 g k_2)	& = -\ln(|(g,e) \cdot 1_e|_h^2) \\
	& = \phi(g).
\end{align*}
\end{proof}

Assume that $h$ is in addition non-negatively curved. Then $\phi$ is a 
$K\times K$-invariant plurisubharmonic function on $G$. Let $u$ be the function on 
$\mathfrak{a}$ defined by 
\[
u(x)= \phi(\exp(x)).
\]
Then Azad and Loeb proved in \cite{AL92} that the function $u$ is convex and 
$W$-invariant.

In particular, since we assumed that the local potentials of singular hermitian 
metrics are locally integrable, the $K\times K$-invariance of $h$ ensures that the 
functions $u$, respectively $ \phi$ take finite values on $\mathfrak{a}$, resp. $G$. 
Indeed, a convex function that takes an infinite value at a point must take an 
infinite value 
on a whole octant starting from that point, and then the corresponding 
$K\times K$-invariant function on $G$ is not locally integrable.

\begin{defn}
We will call $u$ the \emph{convex potential} of $h$. 
\end{defn}

\subsection{Asymptotic behavior of the convex potential}

\subsubsection{A special metric}
\label{BTmetric}

Let us begin by introducing a continuous, $K\times K$-invariant,  reference hermitian   
metric on $L$. We start from the Batyrev-Tschinkel metric 
defined on toric manifolds, and generalize it to build a reference continuous 
metric for any polarized group  compactification $(X,L)$, with convex potential the 
support function of the polytope $2 P(X,L)$.

Given a toric manifold $Z$, equipped with a linearized line bundle $D$, there is a 
natural continuous hermitian metric $h_D$, invariant under the action of the compact 
torus, on $D$, called the Batyrev-Tschinkel metric (see \cite[Section 3.3]{Mai00}). 
If furthermore the line bundle $D$ is ample, then this metric is non-negatively curved, 
and its convex potential is the support function $v$ of the polytope $2 P(Z,D)$. 

Suppose now that $(X,L)$ is a polarized group compactification, and $Z$ is the toric 
submanifold. Denote $L|_Z$ by $D$. Then $P(Z,D)=P(X,L)$ is $W$-invariant, which implies 
that the Batyrev-Tschinkel metric $h_D$ is $W$-invariant. 

We want to extend $h_D$ to a continuous $K\times K$-invariant metric $h_L$ on $X$. 
Define $h_L$ at $\xi \in L_g$ by $|\xi|_{h_L}=|(k_1,k_2)\cdot \xi|_{h_D}$, for 
$(k_1,k_2) \in K\times K$ such that $k_1 g k_2^{-1} \in T$. We need to check that 
this is well defined. Since $h_D$ is $W$-invariant we only need to check that, for 
$t\in T$,  if $(k_1,k_2)\in \mathrm{Stab}_{K\times K}(t)$ then 
$|\xi|_{h_D}=|(k_1,k_2)\cdot \xi|_{h_D}$. But $\mathrm{Stab}_{K\times K}(t)$ acts 
linearly on the line $L_t$, through a character $\chi$. By compacity, 
$|\chi(k_1,k_2)|=1$, so 
$|(k_1,k_2)\cdot \xi|_{h_D}=|\chi(k_1,k_2)\xi|_{h_D}= |\xi|_{h_D}$.

\subsubsection{Asymptotic behavior}

\begin{thm}
\label{Asbe}
The singular hermitian $K\times K$-invariant metrics $h$ with non negative current 
curvature are in bijection with the convex $W$-invariant functions 
$u :\mathfrak{a}\longrightarrow \mathbb{R}$
satisfying the condition that there exists a constant $C_1\in \mathbb{R}$ such that 
\[
u(x) \leq v(x)+C_1
\]
on $\mathfrak{a}$, where $v$ is the support function of the polytope $2 P(X,L)$.
This bijection is obtained by associating to $h$ its convex potential $u$.
Furthermore, $h$ is locally bounded if and only if there exists in addition a constant 
$C_2\in \mathbb{R}$ such that 
\[
v(x)+C_2 \leq \varphi(x) \leq v(x)+C_1.
\]
\end{thm}

\begin{proof}
Let $h$ be a singular hermitian $K\times K$-invariant metric with non negative 
current curvature on $L$. Let $u$ be its convex potential. Recall that $h_L$ denotes 
the reference continuous metric constructed above, and let $\omega_L$ be the 
curvature current of $h_L$. Denote by $\psi$ the potential of $h$ with respect to 
$h_L$. It is an $\omega_L$-psh function on $X$. In particular, $\psi$ is bounded from 
above on $X$.

Denote by $w$ the function on $\mathfrak{a}$ associated to the $K\times K$-invariant 
function $\psi|_G$. Then we see that the function $u-v$ is equal to $w$ 
and thus bounded from above.

If furthermore $h$ is locally bounded then since $h_L$ is also locally bounded, the 
function $w$ is bounded on $X$. So $w=u-v$ is bounded on $\mathfrak{a}$.

Conversely, let $u$ be a convex $W$-invariant function such that $u(x) \leq v(x)+C$.
We choose any reference metric $h_0$ on $L$ that is smooth, positively curved and 
$K\times K$-invariant. Then by the first direction there exist constants $C_1$ and 
$C_2$ such that if $u_0$ is the potential of $h_0$ we have
\[
v(x)+C_2 \leq u_0(x) \leq v(x)+C_1.
\]
Let $\omega_0$ be the curvature form of $h_0$.

Consider the function $w:= u - u_0$. It will be enough to show that the function 
$\psi$ on $G$ corresponding to $w$ extends to an $\omega_0$-psh function on $X$.

First remark that $\psi=\phi-\phi_0$, and by the other direction of the result of 
Azad and Loeb \cite{AL92}, $\phi$ is psh on $G$. The assumption on $u$ implies that 
$w$, and thus $\psi$, are bounded from above. Indeed, we have 
\[
w=u-u_0 \leq v+C-u_0 \leq C-C_2.
\]

A classical result on psh functions is that a psh function extends over an analytic 
subset if and only if it is locally bounded from above. Here, applying that with 
$\psi$ allows to extend $\psi$ to an $\omega_0$-psh function on $X$. The corresponding 
singular hermitian metric $h$ has non-negative curvature, is $K\times K$-invariant, 
and has convex potential $u$.

For locally bounded metrics, one just needs to use the refinement that if 
a psh function is locally bounded then it extends to a bounded psh function.
\end{proof}

\section{Newton bodies}

In this section we introduce a convex body associated to any 
non-negatively curved singular $K\times K$-invariant hermitian metric $h$
on an ample linearized line bundle $L$ on a group compactification $X$.
We first define a convex set associated to any function, which is a natural 
set to consider in the case of convex functions. Applying this construction 
to the convex potential of a hermitian metric yields a convex body that is 
contained in $2P(X,L)$, that will be used to compute the log canonical 
threshold of $h$.

\subsection{Newton set of a function}

\begin{defn}
Let $f$ be a function $\mathfrak{a}\rightarrow \mathbb{R}$, 
and $\sigma$ a closed convex cone in $\mathfrak{a}$.
We call \emph{Newton set} of $f$ the following set in $\mathfrak{a}^*$. 
\[
N_{\sigma}(f):=\{m\in \mathfrak{a}^*; \exists C, \forall x \in \sigma,  f(x)-m(x) \geq C\}.
\]
\end{defn}

In the following, we will simply call cone a closed convex cone.
For any function $f$ and any cone $\sigma$, the Newton set $N_{\sigma}(f)$ is clearly convex.

Recall the definition of  the dual cone $\sigma^{\vee}$ of $\sigma$:
\[
\sigma^{\vee}=\{ m\in \mathfrak{a}^* ; m(x)\geq 0 ~ \forall x \in \sigma \}.
\]
The Newton set $N_{\sigma}(f)$ is by definition  stable under addition 
of an element of the opposite of the dual cone 
$\sigma^{\vee} \subset \mathfrak{a}^*$.
We write this also 
$N_{\sigma}(f)=N_{\sigma}(f)+(-\sigma^{\vee})$
where the plus sign means the Minkowki sum of sets.

\begin{exa}
Let $f$ be the affine function $f(x)=m(x)+c$ where $m\in \mathfrak{a}^*$
and $c$ is a constant. Then $N_{\sigma}(f)=m+(-\sigma^{\vee})$. 
\end{exa}

Let us record the following elementary properties of Newton sets. 

\begin{prop}
\label{easyprop}
Let $f$ and $g$ be two functions on $\mathfrak{a}$ and $c\in \mathbb{R}$. Then 
\begin{itemize}
\item $N_{\sigma}(cf)=cN_{\sigma}(f)$
\item $N_{\sigma}(f+c)=N_{\sigma}(f)$
\item if $f\leq g$ then $N_{\sigma}(f)\leq N_{\sigma}(g)$. 
\item In particular, if for some constants $c_1$ and $c_2$, 
\[
g+c_1\leq f \leq g+c_2
\]
on $\sigma$, then $N_{\sigma}(f)=N_{\sigma}(g)$.
\item Let $\sigma_1$ and $\sigma_2$ be two convex cones such that $\sigma=\sigma_1\cup \sigma_2$, 
then 
\[
N_{\sigma}(f)=N_{\sigma_1}(f)\cap N_{\sigma_2}(f).
\]
\end{itemize} 
\end{prop}

The last property is very helpful when we want to use a fan decomposition.

\begin{exa}
\label{CPLNew}
Let $v:\mathfrak{a} \rightarrow \mathbb{R}$ be a piecewise linear function along a finite fan 
decomposition $\sigma_0 = \cup_{i=1}^N \sigma_i$ where $N\in \mathbb{N}$ and the $\sigma_i$
are convex cones. For $1\leq i\leq N$, denote by $v_i$ the element of $\mathfrak{a}^*$ such that 
$v(x)= v_i(x)$ on $\sigma_i$.
Then 
\[
N_{\sigma}(v)=\bigcap_{i=1}^N (v_i+(-\sigma_i^{\vee})).
\]

If furthermore $v$ is convex, then 
$N_{\sigma}(v)= \mathrm{Conv}\{v_i\}+(-\sigma^{\vee}).$
\end{exa}

\subsection{Newton set of convex functions}

For this subsection only, we will allow convex functions to take the value 
$+\infty$. If $f$ is such a function we define its \emph{domain}
by 
\[
\mathrm{dom}(f):=\{x \in \mathfrak{a} ; f(x)<\infty\}.
\]
We impose however that all functions considered have a non empty domain.
In the rest of the section, we always assume $\mathrm{dom}(f)=\mathfrak{a}$.

The first remark is that the Newton set of a function $f$
on the whole of $\mathfrak{a}$ is the domain of its \emph{Legendre-Fenchel 
transform} (or convex conjugate) $f^*$ defined, for $m\in \mathfrak{a}^*$, by 
\[
f^*(m):=\mathrm{sup}\{m(x)-f(x) ; x\in \mathfrak{a}\}.
\]

Let $\sigma$ be a convex cone, and define the convex function 
$\delta_{\sigma}$ as the indicator function of $\sigma$, \emph{i.e.}
$\delta_{\sigma}(x)=0$ if $x\in \sigma$ and $\delta_{\sigma}(x)=\infty$ 
otherwise.
Then it is not hard to check that 
$N_{\sigma}(f)=N_{\mathfrak{a}}(f+\delta_{\sigma})$.
In other words $N_{\sigma}(f)$ is the domain of the convex conjugate 
of $f+\delta_{\sigma}$. 

We will recall a classical result on convex functions, which allows 
to express the Newton set of a sum as the Minkowski sum of the 
Newton sets of the summands.
First recall the definition of infimal convolution:

\begin{defn}
Let $f$ and $g$ be two convex function. The \emph{infimal convolution}
of $f$ and $g$ is the function $f\square g$ defined, for $x\in \mathfrak{a}$, by 
\[
f\square g(x)=\mathrm{inf}\{f(x-y)+g(y); y\in \mathfrak{a}\}.
\]
\end{defn}

\begin{thm}
\label{Rockafellar}
\cite[Theorem 16.4]{Roc97}
Let $f$ and $g$ be two convex functions on $\mathfrak{a}$, such that 
the relative interiors of the domains of $f$ and $g$ have a point 
in common. Then 
\[
(f+g)^*(m)=f^*\square g^*.
\]
\end{thm}

\begin{prop}
Let $\sigma$ be a convex cone, and $f$ a convex function with 
$\mathrm{dom}(f)=\mathfrak{a}$, then 
\[
N_{\sigma}(f)= N_{\mathfrak{a}}(f)+(-\sigma^{\vee}).
\]
\end{prop}

\begin{proof}
We have seen that $N_{\sigma}(f)$ is the domain of the convex conjugate 
of $f+\delta_{\sigma}$, but by Theorem~\ref{Rockafellar}, this is also 
the domain of the function $f^*\square \delta_{\sigma}^*$. We can apply this 
theorem because the intersection of the domains of $f$ and $\delta_{\sigma}$
is $\sigma$.

The domain of an infimal convolution is the Minkowski sum of the domains 
of the two functions involved,  
so we just need to compute the domain 
of $\delta_{\sigma}^*$.
By definition we check that this is $-\sigma^{\vee}$, and obtain 
the statement.
\end{proof}

\begin{prop}
\label{Newsum}
Let $f$ and $g$ be two convex functions, both with domain
$\mathfrak{a}$, 
and $\sigma$ a convex cone.
Then $N_{\sigma}(f+g)=N_{\sigma}(f)+N_{\sigma}(g)$.
\end{prop}

\begin{proof}
We have, by the previous proposition, 
\[
N_{\sigma}(f+g)= N_{\mathfrak{a}}(f+g)+(-\sigma^{\vee}),
\]
and by the same proof, 
\[
N_{\mathfrak{a}}(f+g)=N_{\mathfrak{a}}(f)+N_{\mathfrak{a}}(g),
\]
so 
\begin{align*}
N_{\sigma}(f+g) & = N_{\mathfrak{a}}(f)+N_{\mathfrak{a}}(g)+(-\sigma^{\vee}) \\
& = N_{\sigma}(f)+N_{\sigma}(g).  
\end{align*}
\end{proof}

\subsection{Newton body of a metric}

Let $X$ be a compactification of $G$, polarized by $L$.
Let $h$ be a $K\times K$-invariant hermitian metric with non negative curvature on $L$, 
and $u$ its convex potential with respect to a fixed left-equivariant trivialization of $L$ 
on $G$, 
which is a function on $\mathfrak{a}$.

\begin{defn}
We will call \emph{Newton body} of $h$ the set 
$N(h):=N_{\mathfrak{a}}(u)$.
\end{defn}

Let $P$ be the polytope corresponding to the polarization $L$.

\begin{exa}
\label{newBT}
Let $h_L$ be the
metric constructed in Section~\ref{BTmetric}.
Its convex potential $v$ is the support function of $2P$, so 
$N(h_L)=2P$, as in Example~\ref{CPLNew}.
Remark that the convex potential of $h_L$ is piecewise linear with respect 
to the opposite of the fan of the toric subvariety.
\end{exa}

\begin{prop}
The Newton body of $h$ is stable under the action of the Weyl group $W$.
\end{prop}

\begin{proof}
Let $u$ be the convex potential of $h$, and let $m\in \mathfrak{a}^*$. 
Suppose that 
\[
u(x)-m(x)\geq C
\]
for some constant $C$ and for all $x\in \mathfrak{a}$.
Let $w\in W$.
By $W$-invariance of $u$, the inequality is equivalent to 
\begin{align*}
C & \leq  u(w\cdot x)-m(x)\\
& \leq u(w\cdot x)-w^{-1}\cdot m(w\cdot x).
\end{align*}

Since $w$ induces a bijection of $\mathfrak{a}$, we get that 
for all $w\in W$, $m\in N(h)$ if and only if $w\cdot m\in N(h)$,
which means that $N(h)$ is $W$-invariant.
\end{proof}

We can finally translate the information about the asymptotic behavior of metrics in terms 
of their Newton bodies.

\begin{prop}
Let $h$ be a $K\times K$-invariant hermitian metric with non negative curvature on $L$.
Then $N(h)\subset 2P$. If in addition $h$ is locally bounded, then $N(h)=2P$. 
\end{prop}

\begin{proof}
Recall from Section~\ref{BTmetric}
that the convex potential $u$ of a 
$K\times K$-invariant hermitian metric $h$ with non negative curvature on $L$
satisfies 
\[
u \leq v+C_2
\]
on $\mathfrak{a}$ for some constant $C_2$, 
and that if $h$ is locally bounded then we have in addition 
\[
v+C_1\leq u
\]
for some constant $C_1$.

Now the result easily follows from Proposition~\ref{easyprop} 
and Example~\ref{newBT}.
\end{proof}

\section{Log canonical thresholds}

In this section, we reduce the computation of the log canonical threshold of a 
$K\times K$-invariant non-negatively curved metric
to an integrability problem involving its convex potential, by using the $KAK$ integration 
formula. We prove an integrability criterion for exponentials of concave functions, with 
respect to the measure $J(x)dx$ appearing in the $KAK$ integration formula, and then use 
it to obtain an expression of the log canonical threshold in terms of the Newton body 
of the metric.

\subsection{Log canonical thresholds on compact manifolds}

In this subsection we consider first $X$ a compact complex manifold that is not necessarily 
a group compactification, and $L$ a line bundle on $X$. 

\begin{defn}
Let $x$ be a point in $X$, and $h$ a hermitian metric on $L$.
The \emph{complex singularity exponent} (or \emph{local log canonical threshold})
of $h$ at $x$, which we denote by $\mathrm{lct}(h,x)$ 
is the supremum of all $c>0$ such that $e^{-c\varphi}$ is integrable 
with respect to Lebesgue measure in a neighborhood of $x$, where
$\varphi$ is the potential of $h$ with respect to a trivialization $s$ of $L$
in a neighborhood of $x$.
\end{defn}

\begin{rem}
If $h$ is a locally bounded metric then on a sufficiently small neighborhood of 
any point, the potential $\varphi$ is a bounded function, so it is integrable.
It means that for any such metric, $\mathrm{lct}(h,x)=\infty$ at any point $x$.
\end{rem}

\begin{defn}
Let $h$ be a hermitian metric on $L$, then the \emph{log canonical threshold} of 
$h$ is defined as 
\[
\mathrm{lct}(h)=\mathrm{inf}_{x\in X}(\mathrm{lct}(h,x)).
\]
\end{defn}

\begin{prop}
\label{lctcompact}
Let $h$ be a singular hermitian metric on $L$, $h_0$ a locally bounded hermitian metric  
on $L$, and $\psi$ the potential of $h$ with respect to $h_0$.   
Let also $dV$ be any smooth volume form on $X$.
Then we have 
\[
\mathrm{lct}(h)=\mathrm{sup}\left\{ c>0;\int_Xe^{-c\psi}dV< \infty \right\}.
\]
\end{prop}

\begin{proof}
Let $x$ be any point in $X$, and $s$ a trivialization of $L$ on a neighborhood $U$ of $x$. 
Up to shrinking $U$, we can assume that the local potential $\varphi_0$ 
of $h_0$ with respect to $s$ is bounded.

Let $\varphi$ be the local potential of $h$ with respect to $s$ and $\psi$ the potential 
of $h$ with respect to $h_0$. Then by definition of $\psi$, 
we have $\psi=\varphi-\varphi_0$ on $U$,
and since $\varphi_0$ is bounded, the integrability of $e^{-c\varphi}$ with respect to 
Lebesgue measure on a 
neighborhood of $x$ is equivalent to the integrability of $e^{-c\psi}$ on the same 
neighborhood.

Furthermore, in the neighborhood of any point $x$ in $X$, the integrability with respect
to Lebesgue measure is equivalent to integrability with respect to a smooth volume form.

The function $\psi$ is defined everywhere on $X$, $e^{-c\psi}$ is positive, and $X$ is 
compact, so 
$e^{-c\psi}$ is integrable with respect to $dV$ in the neighborhood of any point in $X$
if and only if $\int_X e^{-c\psi}dV<\infty$.

Take $0<c<\mathrm{lct}(h)$, then $c<\mathrm{lct}(h,x)$ for all $x\in X$, 
so $\int_X e^{-c\psi}dV<\infty$. This means that 
\[
\mathrm{lct}(h) \leq \mathrm{sup}\left\{ c>0;\int_Xe^{-c\psi}dV< \infty \right\}.
\]

Conversely, if $c>\mathrm{lct}(h)$ then there exists $x\in X$ such that 
$c>\mathrm{lct}(h,x)$ but then $\int_X e^{-c\psi}dV=\infty$, so 
$c\geq \mathrm{sup}\left\{ c>0;\int_Xe^{-c\psi}dV< \infty \right\}$.
Taking the infimum gives the other inequality:
\[
\mathrm{lct}(h) \geq \mathrm{sup}\left\{ c>0;\int_Xe^{-c\psi}dV< \infty \right\}.
\]
This proves the proposition.
\end{proof}

\subsection{Log canonical thresholds on group compactifications}

Let $X$ be a Fano compactification of $G$. Let $L$ be an ample linearized line bundle on $X$.
Using Proposition~\ref{lctcompact}, we reduce the computation of log canonical thresholds to 
integrability conditions on the potentials of metrics, with respect to a smooth volume form. 
Since volume forms do not put weight on the (codimension one) boundary $X\setminus G$, 
we will restrict to integrability conditions on $G$. We want to use, in addition, the 
$KAK$ integration formula, that we recall here : 

\begin{prop}\cite[Proposition 5.28]{Kna02}
Let $dg$ denote a Haar measure on $G$, and $dx$ the Lebesgue measure 
on $\mathfrak{a}$, normalized by the lattice of one parameter subgroups $N$.
Then there exists a constant $C>0$ such that for any 
$K\times K$-invariant, $dg$-integrable function on $G$, 
\[
\int_G f(g)dg= C\int_{\mathfrak{a}^+}J(x)f(\exp(x))dx,
\]
where 
\[
J(x)=\prod_{\alpha \in \Phi^+} \mathrm{sinh}(\alpha(x))^2.
\]
\end{prop}

Let us now derive an integrability criterion with respect to $J$.

\subsection{Integrability criterion}

\subsubsection{Integrability criterion on a cone}

We use the following proposition, obtained by Guenancia in \cite{Gue12}.
It is an analytic proof and generalization of the computation by Howald of the 
log canonical thresholds of monomial ideals. The statement given here is 
slightly different from the statement in \cite{Gue12}, but is in fact equivalent
(see \cite{Del14} for details).

\begin{prop} \cite[Proposition 1.9]{Gue12} \label{Gue12}
Let $f$ be a convex function on $\mathfrak{a}$. 
Assume that $\sigma$ is a smooth polyhedral cone in $\mathfrak{a} = N\otimes \mathbb{R}$.
Then $e^{-f}$ is integrable
on a translate (equivalently on all translates) of $\sigma$ if and only if 0 is in 
the interior of the Newton body of $f$: $0\in \mathrm{Int}(N_{\sigma}(f))$.
\end{prop}

\subsubsection{Integrability with respect to J}

The half sum of the positive roots of $\Phi$ is denoted by $\rho$. 
We want to prove the following integrability criterion, 
with respect to the measure $J(x)dx$.

\begin{prop}
 \label{criterionJ}
Assume that $\mathfrak{a}^+=\bigcup_i \sigma_i$ where each 
$\sigma_i$ is a smooth polyhedral cone of full dimension $r$.
Let $l$ be a function on $\mathfrak{a}$, convex on each cone $\sigma_i$.
Then 
\[
\int_{\mathfrak{a}^+}e^{-l(x)}J(x)dx < +\infty
\]
if and only if 
$4\rho \in \mathrm{Int}(N_{\mathfrak{a}^+}(l))$.
\end{prop}

\begin{lem}
Let $\sigma$ be a smooth full dimensional polyhedral cone in $\mathfrak{a}^+$, 
$l$ be a convex function on $\mathfrak{a}$,
then the following are equivalent:
\begin{itemize}
\item $\int_{\sigma}e^{-l(x)}J(x)dx<\infty$; 
\item $\int_{\sigma}e^{-l(x)+4\rho(x)}dx<\infty$;
\item $4\rho \in \mathrm{Int}(N_{\sigma}(l))$.
\end{itemize}
\end{lem}

\begin{proof}
Writing 
\[
\mathrm{sinh}(\alpha(x))=\frac{e^{\alpha(x)}-e^{-\alpha(x)}}{2}
=\frac{1}{2}e^{\alpha(x)}(1-e^{-2\alpha(x)}),
\]
we get that 
\[
J(x)=\frac{1}{2^{2\mathrm{Card}(\Phi^+)}}e^{2\sum_{\alpha \in \Phi^+}\alpha(x)}
\prod_{\alpha \in \Phi^+}(1-e^{-2\alpha(x)})^2.
\]

For any $x\in \mathfrak{a}^+$ and $\alpha \in \Phi^+$, $\alpha(x)> 0$, so 
$0\leq e^{-2\alpha(x)}< 1$.
This implies $0< \prod_{\alpha \in \Phi^+}(1-e^{-2\alpha(x)})^2 \leq 1$, so 
\[
0< J(x) \leq \frac{1}{2^{2\mathrm{Card}(\Phi^+)}}e^{4\rho(x)}.
\]

This first inequality allows to say that if $\int_{\sigma}e^{-l(x)+4\rho(x)}dx<\infty$
then 
\[
\int_{\sigma}e^{-l(x)}J(x)dx<\infty.
\]

Let us now prove the converse. 
Choose $\gamma$ a point in the interior of $\sigma$.
 Assume that $e^{-l+4\rho}$ is not integrable on $\sigma$.
Then by the usual integrability criterion (Proposition~\ref{Gue12})
$e^{-l+4\rho}$ is also non integrable on $\gamma+\sigma$.

But now, for $x\in \gamma + \mathfrak{a}^+$ and $\alpha\in \Phi^+$, we have 
$\alpha(x)\geq c=\mathrm{min}_{\beta\in \Phi^+}\beta(\gamma)>0$, so 
$0\leq e^{-2\alpha(x)}\leq e^{-2c}<1$, and this implies 
\[
\left(\frac{1-e^{-2c}}{2}\right)^{2\mathrm{Card}(\Phi^+)}e^{4\rho(x)}\leq J(x)
\leq \frac{1}{2^{2\mathrm{Card}(\Phi^+)}}e^{4\rho(x)}.
\]

This gives that 
\begin{align*}
\int_{\sigma}e^{-l(x)}J(x)dx & \geq \int_{\gamma + \sigma}e^{-l(x)}J(x)dx \\
& \geq \int_{\gamma + \sigma}e^{-l + 4 \rho}dx \\
& \geq \infty.
\end{align*}
 
We have then shown the equivalence of the two first points in the lemma. 
By Proposition~\ref{Gue12}
the second point is also equivalent to 
\[
0\in \mathrm{Int}(N_{\sigma}(l-4\rho))=-4\rho+\mathrm{Int}(N_{\sigma}(l)).
\]
Letting $4\rho$ go to the left, we conclude the proof.
\end{proof}

Now we can prove the proposition, just by gluing the parts.

\begin{proof}
Just remark that since the function $e^{-l(x)}J(x)$ is positive and the cones are
full dimensional, $\int_{\mathfrak{a}^+}e^{-l(x)}J(x)dx < +\infty$ if 
and only if 
$\int_{\sigma_i}e^{-l(x)}J(x)dx < +\infty$ for all $i$.

For each of these integrals we can use the lemma, so the necessary and sufficient 
condition becomes 
$4\rho \in \mathrm{Int}(N_{\sigma_i}(l))$ for all $i$, or equivalently 
$4\rho \in \mathrm{Int}(\bigcap_i N_{\sigma_i}(l))$.

To conclude, observe that $N_{\mathfrak{a}^+}(l)=\bigcap_i N_{\sigma_i}(l)$
by Proposition~\ref{easyprop}.
\end{proof}

\subsection{Log canonical thresholds in terms of Newton bodies}

Let $X$ be a Fano compactification of $G$. 
Let $L$ be a 
linearized ample line bundle on $X$, whose associated polytope
is $P$. Denote by $Q$ the polytope associated to the anticanonical bundle $-K_X$. 
Let also $H$ denote the convex hull of all images of $2\rho$ by the Weyl group $W$.

We want to prove the following.
\begin{thm}
\label{exprlct}
Let $h$ be a $K\times K$-invariant hermitian metric with non negative curvature on $L$, then 
\[
\mathrm{lct}(h)=\mathrm{sup}\{c>0; 2H+2cP \subset cN(h)+2Q\}.
\]
\end{thm}

We first introduce some notations.

Let us fix $s_0$ a left $G$-equivariant trivialization of $L$ on $G$ and $s_1$ a left $G$ 
equivariant trivialization of $-K_X$ on $G$. 

Let $u$ be the convex potential of $h$ with respect to the section $s_0$.
Let also $u_0$ be the support function of $P$ and $h_0$ be the corresponding
metric. 
It has locally bounded potentials.
Denote by $\psi$ the potential of $h$ with respect to $h_0$.

Since $X$ is Fano, we can choose $h_1$ a smooth metric on $-K_X$ 
with positive curvature, 
and let $u_1$ be its convex potential with respect to $s_1$.
This choice determines a smooth volume form on $X$, 
which writes, on $G$, 
\[
dV=e^{-u_1}dg
\]
where $dg$ is the Haar measure $s_1^{-1}\wedge \overline{s_1^{-1}}$.     

\begin{rem}
\label{HinQ}
In particular, the integral of this volume form is finite, so applying 
the $KAK$ integration formula this means that 
\[
\int_{\mathfrak{a}^+}e^{-u_1}Jdx<\infty.
\]
By Proposition~\ref{criterionJ}, this implies that 
\[
4\rho \in \mathrm{Int}(N(h_1))=\mathrm{Int}(2Q).
\]
Another way to say that is $H \subset \mathrm{Int}(Q)$.
\end{rem}

\begin{proof}
Using Proposition~\ref{lctcompact}, then restricting to the dense orbit, we get:
\begin{align*}
\mathrm{lct}(h) & = \mathrm{sup}  \left\{ c>0; \int_Xe^{-c\psi}dV<\infty \right\}  \\
            & = \mathrm{sup} \left\{c>0; \int_G e^{-c\psi}dV<\infty \right\}. 
\end{align*}
Since $\psi(\exp(x))=u(x)-u_0(x)$, we can now use the $KAK$ integration formula to write:
\[
\mathrm{lct}(h) = \mathrm{sup} \left\{c>0; \int_{\mathfrak{a}^+} e^{-c(u-u_0)}e^{-u_1}Jdx<\infty \right\}. 
\]

Then Proposition~\ref{criterionJ} gives:
\begin{align*}
\mathrm{lct}(h) & 
	= \mathrm{sup} \left\{ c>0; 
		4\rho \in \mathrm{Int}(N_{\mathfrak{a}^+}(cu-cu_0+u_1)) \right\} \\  
            & = \mathrm{sup} \{c>0; 4\rho \in N_{\mathfrak{a}^+}(cu-cu_0+u_1) \}. 
\end{align*}  

Let $\sigma_i$ be the cones of full dimension in the 
fan subdivision of $\mathfrak{a}^+$ corresponding to $X$
(induced by the fan subdivision of $\mathfrak{a}$ associated to the toric 
subvariety $Z$). 
Then $u_0$ is 
linear on each $-\sigma_i$. We write $u_0^i$ the corresponding element 
of $\mathfrak{a}^*$. 

We have 
\begin{align*}
\mathrm{lct}(h)& = \mathrm{sup} \{c>0; \forall i,~ 4\rho \in N_{-\sigma_i}(cu-cu_0+u_1) \} \\
 & = \mathrm{sup} \{c>0; \forall i,~ 4\rho + cu_0^i \in N_{-\sigma_i}(cu+u_1) \}. \\
\intertext{Recall from Example~\ref{CPLNew} that 
$P=N_{\mathfrak{a}}(u_0)\subset u_0^i+\sigma_i^{\vee}$, so that }
\mathrm{lct}(h)& = \mathrm{sup} \{c>0; \forall i,~ 4\rho + cP \in N_{-\sigma_i}(cu+u_1) \} \\
& = \mathrm{sup} \{c>0; 4\rho+cP \in N_{\mathfrak{a}^+}(cu+u_1) \} \\
            & = \mathrm{sup}\{c>0; 2H+2cP \subset N_{\mathfrak{a}}(cu+u_1)\} 
\end{align*}
by $W$-invariance.

To conclude it remains to remark that both $u$ and $u_1$ are convex, so 
by Proposition~\ref{Newsum},
\[
N_{\mathfrak{a}}(cu+u_1)=cN_{\mathfrak{a}}(u)+N_{\mathfrak{a}}(u_1)=cN(h)+2Q.
\]
\end{proof}

\section{Alpha invariants}

We obtain in this section an expression for the $\alpha$-invariant of a polarized 
group compactification in terms of its polytope. We first give the result for general 
reductive group compactifications, then see how it simplifies when the group is 
semisimple. We then discuss some some examples and how some additional symmetries can 
be taken into account for reductive group compactifications.

\subsection{General formula}

\begin{defn}
Let $X$ be a compact complex manifold, $K$ a compact subgroup of the automorphisms group 
of $X$, 
and $L$ a $K$-linearized line bundle on $X$. 
The \emph{alpha invariant} of $L$ relative to the group $K$, denoted by $\alpha_K(L)$ is the 
infimum of the log canonical thesholds of all $K$-invariant singular hermitian metrics on $L$ 
with 
non negative curvature.
\end{defn}

Let $P$ and $Q$ be two convex bodies in $\mathfrak{a}^*$.
Recall the definition of the \emph{Minkowski difference}: 
\[
Q\ominus P = \{x | x+P\subset Q\}.
\]
Another expression of the Minkowski difference is the following, 
which shows that it is convex if $Q$ is convex:
\[
Q \ominus P = \bigcap_{p\in P} (-p+Q).
\]
If $P_1$, $P_2$ and $Q$ are three convex bodies, then 
$P_1+Q\subset P_2$ if and only if $P_1\subset P_2\ominus Q$.

We can now state our main result.

\begin{thm}
\label{alphared}
Let $(X,L)$ be a polarized compactification of $G$, and $P:=P(X,L)$.
Assume furthermore that $X$ is Fano and let $Q:=P(X,-K_X)$.
Then 
\[
\alpha_{K\times K}(L)=\mathrm{sup} \{ c>0 ; c(P + (-P^W)) \subset Q \ominus H \},
\]
where $P^W$ denotes the subset of $W$-invariant points of $P$.
\end{thm}

\begin{proof}
Let $h$ be any $K\times K$-invariant metric on $L$ with non negative curvature.
The Newton body of $h$ is convex 
and $W$-stable. In particular it contains a $W$-invariant point $p$, 
for example the barycenter of the orbit of any point in $N(h)$.

Denote by $h_p$ the $K\times K$-invariant metric on $L$ with non negative curvature
whose convex potential is the function $x\mapsto p(x)$.
Then $\{p\}=N(h_p) \subset N(h)$, so by the expression of the 
log canonical thresholds from Theorem~\ref{exprlct}, 
$\mathrm{lct}(h)\geq \mathrm{lct}(h_p)$.

Since all such $h_p$ for $p\in 2P^W$ define a singular 
hermitian metric with non-negative curvature, 
this remark allows to write the alpha invariant as
\[
\alpha_{K\times K}(L)=\mathrm{inf}_{p\in 2P^W} \mathrm{lct}(h_p).
\]

Now from the expression of the log canonical threshold we get 
\begin{align*}
\mathrm{lct}(h_p) & = \mathrm{sup} \{ c>0 ; 2H+2cP \subset cN(h_p) + 2Q\} \\
& = \mathrm{sup} \{ c>0 ; -cp+2cP \subset 2Q \ominus 2H \}.
\end{align*}

Then the expression of the alpha invariant further simplifies as
\begin{align*}
\alpha_{K\times K}(L) & = \mathrm{inf}_{p\in 2P^W} 
		\mathrm{sup} \{ c>0 ; -cp+2cP \subset 2Q \ominus 2H \} \\
& = \mathrm{sup} \{ c>0 ; \forall p\in 2P^W,  -cp+2cP \subset 2Q \ominus 2H \} \\
& = \mathrm{sup} \{ c>0 ; 2cP + (-2cP^W) \subset 2Q \ominus 2H \}.\\
\intertext{Dividing by two yields}
& =\mathrm{sup} \{ c>0 ; c(P + (-P^W)) \subset Q \ominus H \}
\end{align*}
which is the expression in the statement of the Theorem.
\end{proof}

\begin{rem}
In the toric case, we recover our previous 
computation \cite{Del14}: 
\[
\alpha_{(\mathbb{S}^1)^n}(L)=\mathrm{sup} \{ c>0 ; c(P + (-P)) \subset Q \}.
\]
\end{rem}

\subsection{Semisimple case}

The alpha invariant of an ample line bundle on a Fano compactification
of a semisimple group can be easily expressed in terms of the polytope 
associated to $L$ as an inradius between two convex bodies.

\begin{defn}
The \emph{inradius} of $Q$ with respect to $P$ is the number:
\[
\mathrm{inr}(P,Q) := \mathrm{sup}\{c \geq 0| \exists x, ~ x+cP\subset Q\}.
\]
\end{defn}

\begin{cor}  
\label{alphassinr}
Assume that $G$ is a semisimple group. Then 
\[
\alpha_{K\times K}(L)=\mathrm{inr}(P,Q\ominus H).
\]
\end{cor}

\begin{proof}
If $G$ is semisimple, we have $P^W=\{0\}$. In fact, the metric $h_0$ whose convex potential is 
the zero function satisfies 
\begin{align*}
\alpha_{K\times K}(L) & = \mathrm{lct}(h_0) \\
                     & = \mathrm{sup}\{c>0; cP \subset Q\ominus H\}.
\end{align*}
And this is equal to the inradius $\mathrm{inr}(P,Q\ominus H)$. 

Indeed, one inequality is trivial: $\mathrm{inr}(P,Q\ominus H)\geq \alpha_{K\times K}(L)$.
Conversely, assume $c\leq \mathrm{inr}(P,Q\ominus H)$, i.e. there exists 
an $x\in \mathfrak{a}^*$ such that 
\[
x+cP\subset Q \ominus H.
\]
Then since $P$ and $Q\ominus H$ are stable under $W$-action, we also have 
\[
\forall w\in W, ~ w\cdot x +cP\subset Q \ominus H.
\]
Convexity and the fact that the barycenter of the $W$-orbit of $x$ is $0$ imply 
that $cP\subset Q \ominus H$, so $c\leq  \alpha_{K\times K}(L)$.
We have thus proved the other inequality 
$\mathrm{inr}(P,Q\ominus H)\leq \alpha_{K\times K}(L)$.
\end{proof}

\begin{rem}
In the case of reductive groups, the alpha invariant is not an inradius, 
but we can bound it from above by an inradius:
\[
\alpha_{K\times K}(L) \leq \mathrm{inr}(P+(-P)^W,Q\ominus H).
\]
\end{rem}

\subsection{Additional symmetries}

If the polytopes $P$ and $Q$ admit additional common symmetries, then the 
value of the alpha invariant can be improved. Indeed, the symmetries 
of $Q$ translate to a finite subgroup $O$ of the automorphisms group of the variety $X$, 
and if $P$ is stable under these symmetries, then it is linearized by $O$.
We can thus consider the alpha invariant with respect to the bigger group generated
by $K\times K$ and $O$, that we denote $K_O$.

We then have, adapting the proof of Theorem~\ref{alphared}, 
\[
\alpha_{K_O}(L)=\mathrm{sup} \{ c>0 ; c(P + (-P^{\left<W,O\right>})) \subset Q \ominus H \}.
\]
In particular, if the only fixed point under $\left<W,O\right>$ is the origin, 
then just as in the semisimple case, we get 
\[
\alpha_{K_O}(L)=\mathrm{inr}(P,Q\ominus H).
\]

\subsection{Examples}

Let us compute the $\alpha$ invariant of the anticanonical line bundle for some
wonderful compactifications of semisimple groups. First remark that in this case 
we have $P=Q$ and can rewrite the expression of the invariant as :
\[
\alpha_{K\times K}(X,-K_X)=\mathrm{sup}\{c>0; H\subset (1-c)Q\},
\]
or, by $W$-invariance, 
\[
\alpha_{K\times K}(X,-K_X)=\mathrm{sup}\{c>0; 2\rho\in (1-c)Q^+\}.
\]

\begin{rem} 
It is interesting to notice that this quantity appeared in the determination of some 
volume asymptotics by Chambert-Loir and Tschinkel. 
If $\sigma$ denotes the quantity the authors compute in the examples of compactifications 
of semisimple groups \cite[Section 5.3]{CLT10}, 
we have $\sigma=1-\alpha_{K\times K}(X,-K_X)$ if the polytope considered is the anticanonical 
polytope of a Fano compactification. This is because their computation 
in this special case is equivalent to a computation of the log canonical threshold of a 
metric on the anticanonical line bundle with constant convex potential. 
\end{rem} 

For wonderful compactifications of semisimple adjoint groups, 
the polytope of the anticanonical line bundle $Q$ is determined 
by the root system. Indeed, recall that it is the convex hull 
of the images by $W$ of the weight 
$2\rho + \sum_{i=1}^r \alpha_i$
where the $\alpha_i$ are the simple roots of $\Phi^+$.

In particular, when $G= (\mathrm{PSL}^2(\mathbb{C}))^n$, 
for any $n\geq 1$, the simple roots are the same as the positive roots, 
so $Q=2H$.

\begin{cor}
Let $X$ be the wonderful compactification of $(\mathrm{PSL}_2(\mathbb{C}))^n$, then 
\[
\alpha_{K\times K}(-K_X)=\frac{1}{2}.
\]
\end{cor}

\begin{proof}
Applying Corollary~\ref{alphassinr}
gives 
\[
\alpha_{K\times K}(-K_X)=\mathrm{inr}(2H,H)=\frac{1}{2}.
\]
\end{proof}

More generally for type $A_n$, choosing an appropriate ordering 
of the simple roots $\alpha_1, \ldots, \alpha_n$, we can write the positive roots
as 
\[
\alpha_i+\alpha_{i+1}+\cdots +\alpha_j
\] 
for each pair $(i,j)$ with 
$1\leq i \leq j \leq n$.
We see then that the coefficient of $\alpha_k$ in the sum of positive roots 
$\sum_{l=1}^n \alpha_l$ is equal to the cardinal of the set
$\{(i,j) ; 1\leq i \leq k \leq j\leq n\}$.
This is $k(n-k+1)$.
Adding the sum of simple roots, we see that the $k^{th}$-coordinate of the vertex defining
the polytope of the wonderful compactification of $\mathrm{PSL}_{n+1}(\mathbb{C})$ in the basis of 
simple roots is $1+k(n-k+1)$.

Then from our result, the alpha invariant is easily seen to be the maximum of all 
$c>0$ such that for each $k$, $c(1+k(n-k+1))\leq 1$.
We deduce the following value for the alpha invariant.

\begin{cor}
Let $X$ be the wonderful compactification of $\mathrm{PSL}_{n+1}(\mathbb{C})$, then 
\[
\alpha_{K\times K}(-K_X)=
\frac{1}{1+\lceil \frac{n}{2} \rceil (\lfloor \frac{n}{2} \rfloor +1) }.
\]
\end{cor}

\bibliographystyle{alpha}
\bibliography{biblio}

\end{document}